\documentclass[11pt]{amsart}

\usepackage{input}
\usepackage{tikz-cd}
\usepackage{extarrows}

\newtheorem{thm}{Theorem}[section]
\newtheorem{lemma}[thm]{Lemma}
\newtheorem{proposition}[thm]{Proposition}
\newtheorem{corollary}[thm]{Corollary}

\theoremstyle{definition}
\newtheorem{dfn}[thm]{Definition}

\newtheorem{remk}[thm]{Remark}

\renewcommand{\H}{\mathbb{H}}

\DeclareMathOperator{\Aut}{Aut}

\DeclareMathOperator{\Tors}{Tors}

\DeclareMathOperator{\Hom}{Hom}

\newcommand{\f}[5]{
\begin{array}{rcl}
#1\colon #2 & \longrightarrow & #3 \\
#4 & \longmapsto & #5 \\
\end{array}
}

\newcommand{\cL}{{\mathcal L}}

\newcommand{\cM}{{\mathcal M}}

\newcommand{\cE}{{\mathcal E}}
\newcommand{\cA}{{\mathcal A}}
\newcommand{\cT}{{\mathcal T}}

\newcommand{\ul}[1]{\underline{#1}}
\newcommand{\wt}[1]{\widetilde{#1}}
\newcommand{\wh}[1]{\widehat{#1}}
\newcommand{\ov}[1]{\overline{#1}}

\DeclareMathOperator{\rank}{rank}

\newcommand{\ovLn}{{\ov{\cL_N}}}

\title[]{Eigenspace decomposition of Mixed Hodge Structures on Alexander Modules.}

\author{Eva Elduque}
\address{Departamento de Matem\' aticas, Universidad Aut\' onoma de Madrid, 28049 Madrid (Spain).}
\email {eva.elduque@uam.es}\urladdr{http://matematicas.uam.es/~eva.elduque}
\author{Mois\'es Herrad\'on Cueto}
\address{Department of Mathematics, Louisiana State University, 303 Lockett Hall, Baton Rouge, LA 70803, USA.}
\email {moises@lsu.edu}\urladdr{http://www.math.lsu.edu/~moises}



\begin{document}

\date{\today}

\begin{abstract}
In previous work jointly with Geske, Maxim and Wang, we constructed a mixed Hodge structure (MHS) on the torsion part of Alexander modules, which generalizes the MHS on the cohomology of the Milnor fiber for weighted homogeneous polynomials. The cohomology of a Milnor fiber carries a monodromy action, whose semisimple part is an isomorphism of MHS. The natural question of whether this result still holds for Alexander modules was then posed. In this paper, we give a positive answer to that question, which implies that the direct sum decomposition of the torsion part of Alexander modules into generalized eigenspaces is in fact a decomposition of MHS. We also show that the MHS on the generalized eigenspace of eigenvalue $1$ can be constructed without passing to a suitable finite cover (as is the case for the MHS on the torsion part of the Alexander modules), and compute it under some purity assumptions on the base space.
\end{abstract}
\maketitle


\setlength{\headheight}{15pt}

\thispagestyle{fancy}
\renewcommand{\headrulewidth}{0pt}

\section{Introduction}\label{s:intro}

Let $f:U\rightarrow \C^*$ be an algebraic map, where $U$ is a smooth connected complex algebraic variety. Moreover, assume that $f$ induces an epimorphism $f_*:\pi_1(U)\rightarrow \Z$ in fundamental groups.

The pair $(U,f)$ determines an infinite cyclic cover $\pi:U^f\rightarrow U$, that is, a covering space with deck group isomorphic to $\Z$. Indeed, $\pi$ is the pullback of the exponential map by $f$, as shown in the following fiber product diagram.
$$
\begin{tikzcd}
U^f\subset U\times \C\arrow[r,"f_\infty"] \arrow[d,"\pi"]\arrow[dr,phantom,very near start, "\lrcorner"]&
 \C \arrow[d,"\exp"] \\
 U\arrow[r,"f"] &
 \C^*.
\end{tikzcd}
$$
Under this presentation $U^f$ is embedded in $U\times \C$:
\[
U^f = \{ (x,z)\in U\times \C\mid f(x) = e^z\}.
\]

Let $k$ be $\Q$, $\R$ or $\C$, and let $R=k[t^{\pm 1}]$. By analogy with knot theory, the $R$-module $H_j(U^f;k)$ is called the $j$-th (homology) {\it $k$-Alexander module} of the pair $(U,f)$. The $R$-module structure is given as follows: multiplication by $t$ is given by the action of the generator $(x,z)\mapsto(x,z+2\pi i)$ of the deck transformations of $\pi:U^f\rightarrow U$. Since $U$ is homotopy equivalent to a finite CW-complex, $H_j(U^f;k)$ is a finitely generated $R$-module for each integer $j$. However, $H_j(U^f;k)$ is not a finite dimensional $k$-vector space in general (i.e. $H_j(U^f;k)$ is not necessarily a torsion $R$-module).

In \cite{EGHMW}, a canonical and functorial $k$-mixed Hodge structure (MHS for short) was constructed on  $\Tors_R H_i(U^f;k)$, for $k=\Q,\R,\C$. The theorem below summarizes some of its properties.

\begin{thm}\label{thm:summary}
Let $U$ be a smooth connected algebraic variety, and let $f:U\rightarrow\C^*$ be an algebraic map inducing an epimorphism on fundamental groups. Let $H_j(U^f;k)$ be the $j$-th homology $k$-Alexander module, endowed with the MHS from \cite{EGHMW}. Then,
\begin{enumerate}
\item\label{part:N} There exists $N\in \N$ such that $(t^N-1)$ acts nilpotently on $\Tors_R H_j(U^f;k)$, for all $j$ (\cite[Proposition 1.4]{budurliuwang}, \cite[Proposition 2.6.1]{EGHMW}).
\item\label{part:log} Multiplication by $\log(t^N)$ (seen as a power series in $(t^N-1)$) induces a MHS morphism from $\Tors_R H_j(U^f;k)$ to its $-1$st Tate twist $\Tors_R H_j(U^f;k)(-1)$ for all $j$ (\cite[Theorem 1.0.2]{EGHMW}).
\item\label{part:cover} The covering space map $\pi:U^f\rightarrow U$ induces a MHS morphism $H_j(\pi):\Tors_R H_j(U^f;k)\rightarrow H_j(U;k)$ for all $j$, where $H_j(U;k)$ is endowed with the dual of Deligne's MHS on $H^j(U;k)$ (\cite[Theorem 1.0.3]{EGHMW}).
\item\label{part:t} Multiplication by $t$ is a MHS automorphism of $\Tors_R H_j(U^f;k)$ if and only if $\Tors_R H_j(U^f;k)$ is a semisimple $R$-module (\cite[Theorem 1.0.5]{EGHMW}).
\end{enumerate}
\end{thm}

Part~(\ref{part:t}) of Theorem~\ref{thm:summary} is particularly helpful for understanding the MHS on $\Tors_R H_j(U^f;k)$. Indeed, if $\Tors_R H_j(U^f;k)$ is semisimple, then multiplication by any polynomial in $t$ with $k$-coefficients is a MHS morphism. Since the kernel of a MHS morphism is a sub-MHS of the domain, we obtain the following direct sum decomposition.

\begin{corollary}[See Corollary 7.0.4 in \cite{EGHMW}]\label{cor:eigenspaces}
Suppose that $\Tors_R H_j(U^f;\Q)$ is a semisimple $R$-module. Let $N$ be as in part~(\ref{part:N}) of Theorem~\ref{thm:summary}, and let $k = \Q$, $\R$ or $\C$. For $g\in k[x]$, let $\left(\Tors_R H_j(U^f;k)\right)_{g}$ be the kernel of $g(t)\colon \Tors_R H_j(U^f;k)\rightarrow \Tors_R H_j(U^f;k)$. Then, the inclusions give us the following $k$-MHS isomorphism.
$$
\Tors_R H_j(U^f;\Q)\cong \bigoplus_g \left(\Tors_R H_j(U^f;\Q)\right)_{g},
$$
where the sum runs over the monic irreducible factors of $x^N-1$ in $k[x]$.
\end{corollary}

$\Tors_R H_j(U^f;k)$ is known to be semisimple in the following cases.
\begin{enumerate}
\item\label{part:first} $j\leq 1$ (\cite[Corollary 7.4.2]{EGHMW}).
\item\label{part:proper} $f$ is proper (\cite[Theorem 8.0.1]{EGHMW}).
\item\label{part:trans} $U$ is the affine complement of a hypersurface $H\subset\C^n$ transversal to infinity, and $f$ is the reduced defining polynomial of the hypersurface (\cite{Max06}). This can be generalized to certain hypersurface arrangement complements in a smooth projective complete intersection, under certain transversality conditions (see \cite{DL} for the precise setting).
\end{enumerate}

In all three cases, the techniques are vastly different: the proof of case~\eqref{part:first} uses the MHS from Theorem~\ref{thm:summary}, case \eqref{part:proper} uses the Beilinson-Bernstein-Deligne decomposition theorem \cite{BBD} and case \eqref{part:trans} uses transversality together with a Lefschetz type argument to reduce to the case when $f$ is homogeneous and thus has a global Milnor fiber.

We do not know of any example where $\Tors_R H_j(U^f;\Q)$ is not semisimple in our setting, where $U$ is smooth and $f$ is a map to $\C^*$. This lack of examples is mainly due to the fact that higher Alexander modules are harder to compute than the first (see case~(\ref{part:first}) above), the reason being that Fox free calculus is not available for computing higher Alexander modules. We expect such non semisimple examples to exist. In fact, if $f$ is taken to be a continuous map from $U$ to $S^1$, or if $U$ is allowed to have singularities, non-semisimple examples of $\Tors_R H_1(U^f;\Q)$ were recently found by Libgober \cite{Li21} and Maxim, Wang and the authors \cite{EHMW}, respectively.

Given the lack of evidence that semisimplicity holds in general, it is natural to ask if a generalization of Corollary~\ref{cor:eigenspaces} can exist for the torsion parts of all homology Alexander modules, without the semisimplicity assumption. For that purpose, one needs to find an isomorphism of MHS from $\Tors_R H_j(U;\Q)$ to itself that can play the role that multiplication by $t$ plays in Corollary~\ref{cor:eigenspaces}. We find one such MHS isomorphism in the main result of this paper, as follows.

\begin{thm}\label{thm:tss}
Let $t=t_{ss}  t_{u}$ be the Jordan-Chevalley decomposition of $t$ acting on $\Tors_R H_j(U^f;\Q)$ as the product of a semisimple (i.e. diagonalizable) operator and a unipotent (i.e. $t_u-\id$ is nilpotent) operator that commute with each other. Then,
$$
t_{ss}  \colon \Tors_R H_j(U^f;\Q)\to \Tors_R H_j(U^f;\Q)
$$
is a MHS isomorphism for all $j$.
\end{thm}

Theorem~\ref{thm:tss} gives an affirmative answer to \cite[Question 3]{EGHMW}, which asked if $t_{ss}$ was a MHS morphism. This question was motivated by the corresponding results for
the semisimple part of the monodromy operator acting on the cohomology of the Milnor fiber (see \cite[Section 14]{NA}), and, respectively, on the cohomology of the generic fiber of a proper family $f\colon U \rightarrow\Delta^*$ over a punctured disc (see \cite[Th\' eor\`eme 15.13]{NA}). In fact, the MHS on the torsion part of Alexander modules can be thought of as a global analogue of the MHS on the Milnor fiber, which makes this a very natural question. More precisely  it was shown in \cite[Corollary 7.2.4]{EGHMW} that, if $U=\C^n\backslash f^{-1}(0)$ and $f$ is a weighted homogeneous polynomial on $n$ variables, then the inclusion of the (global) Milnor fiber $F=f^{-1}(1)$ into $U$ lifts to an inclusion $F\hookrightarrow U^f$ which is a homotopy equivalence, and it induces a MHS isomorphism $H_j(F;k)\cong H_j(U^f,k)=\Tors_R H_j(U^f,k)$. 

As an immediate consequence of Theorem~\ref{thm:tss}, we get the following direct sum decomposition statement, which generalizes Corollary~\ref{cor:eigenspaces}. 

\begin{corollary}\label{cor:geneigenspaces}
Let $N$ be as in part~(\ref{part:N}) of Theorem~\ref{thm:summary}, and let $k = \Q$, $\R$ or $\C$. For $g\in k[x]$, let $\left(\Tors_R H_j(U^f;k)\right)_{g}$ be the kernel of $g(t_{ss}  )\colon \Tors_R H_j(U^f;k)\rightarrow \Tors_R H_j(U^f;k)$. Then, the inclusions give us the following $k$-MHS isomorphism.
$$
\Tors_R H_j(U^f;\Q)\cong \bigoplus_g \left(\Tors_R H_j(U^f;\Q)\right)_{g},
$$
where the sum runs over the monic irreducible factors of $x^N-1$ in $k[x]$.
\end{corollary}
\begin{remk}\label{remk:decomp}
Note that $\ker g(t_{ss}  )=\ker g(t)^m$ for big enough $m$, so the decomposition in Corollary~\ref{cor:geneigenspaces} is the decomposition into generalized eigenspaces (over $k$).
\end{remk}

\begin{remk}
In the local Milnor fiber setting, the fact that the semisimple part of the monodromy is a MHS isomorphism yields a powerful invariant: the spectrum of a hypersurface singularity, first defined by Steenbrink \cite{SteenbrinkSpectrum}. It is a polynomial with fractional exponents, and the coefficients are defined using the Hodge filtration of the cohomology of the Milnor fiber at a singular point in each of the different generalized eigenspaces under the action of $t_{ss}  $. The spectrum is related to other invariants, such as the jumping coefficients of multiplier ideals \cite{budurspectrum,budursaito}, the V-filtration of Malgrange and Kashiwara \cite{budursaito2}, and Bernstein-Sato polynomials \cite{ELSV}. Also, while the Betti numbers of the Milnor fiber of a central hyperplane arrangement are not known to be combinatorial invariants, its spectrum is \cite{budursaito2,budursaito}. Similarly, we hope that Theorem~\ref{thm:tss} and its proof shed some light on possible global analogues of the spectrum.
\end{remk}

The generalized eigenspace decomposition of Corollary~\ref{cor:geneigenspaces} motivates the study of the direct summands appearing in it. A particularly interesting one is $\left(\Tors_R H_j(U^f;k)\right)_1$, defined to be the generalized eigenspace corresponding to the eigenvalue $1$. First of all, $\Tors_R H_j(U^f;k)$ is always an isomorphic MHS  to $\left(\Tors_R H_j(U'^{f'};k)\right)_1$ for suitably chosen $(U',f')$ (see Lemma~\ref{233}). Secondly, the definition of the MHS on $\Tors_R H_j(U^f;k)$ requires the explicit knowledge of $N$ from Theorem~\ref{thm:summary}, part \eqref{part:N}, as it makes use of an $N$-sheeted covering space of $U$. This specific value of $N$ can be very hard to obtain in practice if $j>1$. However, we prove that to compute the MHS on $(\Tors_R H_j(U^f;k))_1$, knowledge of $N$ (or a finite cover of $U$) is not needed (see Remark~\ref{remk:noN} and Corollary~\ref{cor:noCoverNeeded}). We hope that this result will lead to computational techniques for the MHS on this generalized eigenspace.

To illustrate this point, under some purity conditions on the MHS on the cohomology of $U$, we are able to prove the following purity result for $\left(\Tors_R H_j(U^f;k)\right)_1$.
\begin{corollary}\label{cor:formal1}
Let $U$ be a smooth connected complex algebraic variety, let $f:U\rightarrow \C^*$ be an algebraic map inducing an epimorphism on fundamental groups, and let $r\geq 1$. Assume that for every integer $k\leq r$, $H^k(U)$ is pure of weight $2k$. Then, for all $j\leq r-1$:
\begin{itemize}

\item The MHS summand $(\Tors_R H_j(U^f;\Q))_1$ of $\Tors_R H_j(U^f;\Q)$ is pure of type $(-j,-j)$.
\item $(\Tors_R H_j(U^f;\Q))_1$ has dimension
\[
\sum_{l=0}^j (-1)^{l+j}\left(\dim_{\Q}H_l(U;\Q) - \rank_{\Q[t^{\pm 1}]}H_l(U^f;\Q)\right).
\]
\end{itemize}
\end{corollary}

\begin{remk}[Examples of purity]\label{remk:pure}
Corollary~\ref{cor:formal1} for $r=\infty$ applies to hyperplane arrangement complements in $\C^n$ \cite{shapiro}, and toric arrangement complements in $(\C^*)^n$ \cite[Theorem 3.8]{dupont}. 
\end{remk}

The paper is structured as follows. In Section~\ref{s:background}, we recall the construction of the MHS on $\Tors_R H_j(U^f;k)$, some of its important properties, and set notation for the rest of the paper. Section~\ref{s:semisimple} is devoted to the proof of Theorem~\ref{thm:tss}. In Section~\ref{s:eigen1}, we prove that the MHS on $\left(\Tors_R H_j(U^f;k)\right)_1$ does not depend on finite covers of $U$. Finally, in Section~\ref{s:formal}, we prove Corollary~\ref{cor:formal1} using results of Dupont \cite{dupont} and Budur-Liu-Wang \cite{budurliuwang}, discuss the formula for the dimension of $\left(\Tors_R H_j(U^f;k)\right)_1$ (Remark~\ref{remk:rank}) in the cases of Remark~\ref{remk:pure}, and talk about the MHS on Alexander modules of hyperplane arrangement complements in more detail.

\section*{Acknowledgements}
The authors would like to thank Christian Geske, Lauren\c{t}iu Maxim and Botong Wang, since this paper originated from an open problem posed in our joint paper \cite{EGHMW} and from the many fruitful discussions that we have had regarding this topic. The first author would also like to thank Mircea Musta\c{t}\u{a} and Alex Suciu for useful conversations. Both authors were partially supported by individual AMS-Simons Travel Grants during the preparation of this paper.

\section{Background}\label{s:background}

\subsection{Mixed Hodge structures over \texorpdfstring{$\Q$, $\R$ and $\C$}{Q, R and C}}
We will recall the definition of Hodge structures (see \cite{elzein2013mixed}), including the simpler, but less prevalent, $\C$-Hodge structures. All vector spaces in the definitions are finite dimensional, and all filtrations are biregular, i.e. for every filtration on a vector space $V$, there are pieces that equal $0$ and $V$.
\begin{dfn}
Let $w\in \Z$. In this definition, let $k=\Q$ or $\R$.
\begin{itemize}
\item A $\C$-Hodge structure of weight $w$ is a triple $(V,F,F')$ consisting of a $\C$-vector space $V$ and two decreasing filtrations $F,F'$, such that $V = \bigoplus_{p+q = w} F^p\cap F'^q$.
\item A $k$-Hodge structure of weight $w$ is a pair $(V,F)$ consisting of a $k$-vector space and a decreasing filtration $F$ of $V\otimes_k\C$, such that $(V\otimes_k \C,F,\overline F)$ is a $\C$-Hodge structure of weight $w$, where $\ov \cdot$ denotes complex conjugation.
\item A $\C$-mixed Hodge structure (MHS) is a quadruple $(V,W,F,F')$ consisting of a $\C$-vector space $V$, an increasing filtration $W$ and two decreasing filtrations $F,F'$, such that for every $w\in \Z$, $(\Gr_w^W V,\Gr_w^W F,\Gr_w^W F')$ is a Hodge structure of weight $w$, and $\Gr_w^W F$ denotes the filtration induced by $F$ on $\Gr_w^W V$ (similarly for $F'$).
\item A $k$-MHS is a triple $(V,W,F)$ consisting of a $k$-vector space $V$, an increasing filtration $W$ of $V$ and a decreasing filtration $F$ of $V\otimes_k\C$, such that for every $V$, $F$ induces on $\Gr^W_w V$ a Hodge structure of weight $w$.
\item For all of the above, a morphism is a linear map defined over the base field and preserving the filtrations.
\end{itemize}
\end{dfn}

\begin{remk}\label{remk:QRC}
Let us collect the essential facts that we will use. The following statements are straightforward. Recall that for a field extension $k\subseteq k'$, a $k$-structure on a $k'$-vector space $V$ is a sub-$k$-vector space $V_{k}$ such that $V = V_{k}\otimes_{k} k'$. If $L\subset V$ is a $k'$-subspace, we say that it is defined over $k$ if it equals $L_k\otimes_k k'$ for some sub-$k$-vector space $L_k\subseteq V_k$, which is necessarily unique.
\begin{itemize}

\item Tensoring induces forgetful functors from MHS over $\Q$ to $\R$, and from $\R$ to $\C$.
\item A $\Q$-MHS is equivalent to an $\R$-MHS $(V,W,F)$ with the additional data of a $\Q$-structure on $V$ such that $W$ is defined over $\Q$.
\item A $\Q$ (resp. $\R$)-MHS is equivalent to a $\C$-MHS $(V,W,F,F')$ with the additional data of a $\Q$ (resp. $\R$)-structure on $V$ such that $W$ is defined over $\Q$ (resp. $\R$) and $F'=\ov F$.
\item Let $k,k'\in \{\Q,\R,\C\}$ and $k\subset k'$. Let $V_1,V_2$ be two $k$-MHS's. Tensoring induces a bijection
\[
\Hom_{k-\mathrm{MHS}}(V_1,V_2) \cong 
\Hom_{k'-\mathrm{MHS}}(V_1 \otimes_k k',V_2 \otimes_k k') \cap \Hom_{k-\mathrm{linear}}(V_1,V_2),
\]
i.e. a linear map is a MHS morphism if and only if its base change is a MHS morphism.
\end{itemize}
\end{remk}

\subsection{Alexander modules and notation}

Let $\cL$ be the rank 1 local system of $R$-modules on $U$ defined as $\cL=\pi_!\underline k_{U^f}$, which has monodromy representation
$$
\begin{array}{ccc}
\pi_1(U)&\longrightarrow & \Aut_R(R)\\
\gamma& \longmapsto & \cdot t^{f_*(\gamma)}.
\end{array}
$$
Then $H_j(U;\cL)$ is canonically isomorphic as an $R$-module to $H_j(U^f;k)$ \cite[Theorem 2.1]{KL}, and $\left(\Tors_R H_j(U;\cL)\right)^{\vee_k}$ is canonically isomorphic as an $R$-module to $\Tors_R H^{j+1}(U^f;\ov\cL)$ for all $j$ \cite[Proposition 2.4.1]{EGHMW}, where $\vee_k$ denotes the dual as a $k$-vector space, and $\ov\cL\coloneqq\cL\otimes_{t\mapsto t^{-1}}R$ denotes the conjugate local system to $\cL$. The MHS on $\Tors_R H_j(U^f;k)$ from Theorem~\ref{thm:summary} is the dual MHS of the MHS on $\Tors_R H^{j+1}(U^f;\ov{\cL})$ constructed in \cite[Theorem 5.4.10]{EGHMW} through these isomorphisms. This discussion prompts the following definition.

\begin{dfn}
The $j$-th cohomology Alexander module of $(U,f)$ is $H^j(U,\ov\cL)$.
\end{dfn}

\begin{remk}\label{remk:dual}
By the discussion above, Theorem~\ref{thm:summary} and Corollary~\ref{cor:eigenspaces} have a cohomological analogue for $\Tors_R H^{j+1}(U;\ov\cL)$ by taking $k$-duals. The results in this paper can also be stated for both homology and cohomology Alexander modules using the canonical $R$-module isomorphism $\left(\Tors_R H_j(U;\cL)\right)^{\vee_k}\cong \Tors_R H^{j+1}(U^f;\ov\cL)$ (see Corollary~\ref{cor:dual}).
\end{remk}

In Sections~\ref{s:unip} through \ref{s:MHS}, we recall how the MHS on $\Tors_R H^{j+1}(U;\ov\cL)$ was constructed in \cite{EGHMW} (which is also summarized in \cite{EGHMWSurvey}). All the claims that appear were proved in \cite{EGHMW}. This is just a summary of the construction, for the purposes of fixing notation.

\subsection{Step 1: reducing to the unipotent case.}\label{s:unip}
Let $N$ such that $(t^N-1)^m$ acts as $0$ on $\Tors_R H^{j}(U;\ov\cL)$ for some $m$ and all $j$, given by Theorem~\ref{thm:summary}, part (\ref{part:N}). Consider the following pullback diagram:
$$
\begin{tikzcd}
 U_N = \{(x,z)\in U\times \C^*\mid f(x) = z^N \}\arrow[d,"p"]\arrow[dr,phantom,very near start, "\lrcorner"] \arrow[r,"f_N"] & \C^*\arrow[d,"z\mapsto z^N"] \\
 U \arrow[r,"f"] & \C^*.
\end{tikzcd}
$$
Here $p$ is an $N$-sheeted cyclic cover and this is a diagram of algebraic maps between smooth varieties. We can then define, as in Section~\ref{s:intro},  $(U_N)^{f_N}$, $(f_N)_\infty$, $\pi_N$ and $\cL_N$  for the map $f_N\colon U_N\to \C^*$. We also define:
\[
\f{\theta_N}{U^f}{(U_N)^{f_N}}{U\times \C\ni (x,z)}{(x,e^{z/N}, z/N)\in U_N\times \C\subset U\times \C^* \times \C,}
\]
which fits into the following commutative diagram:
\begin{equation}\label{eq:UN}
\begin{tikzcd}[column sep = 7em]
U^f \arrow[d,"f_\infty"]\arrow[r,"\sim"',"\theta_N", dashed] 
\arrow[rrr,"\pi",rounded corners,to path=
{ --([yshift = 0.7em]\tikztostart.north)
-- ([yshift = 0.7em]\tikztotarget.north)\tikztonodes
-- (\tikztotarget)}] 
&
(U_N)^{f_N} \arrow[dr,phantom,very near start, "\lrcorner"] \arrow[d,"(f_N)_\infty"]\arrow[r,"\pi_N"]\arrow[dr,phantom,very near start] &
U_N \arrow[d,"f_N"] \arrow[dr,phantom,very near start, "\lrcorner"] \arrow[r,"p"]\arrow[dr,phantom,very near start] &
U \arrow[d,"f"] \\
\C \arrow[r,"z\mapsto\frac{z}{N}"]\arrow[rrr,"\exp", rounded corners,to path=
{ --([yshift = -0.7em]\tikztostart.south)
-- ([yshift = -0.7em]\tikztotarget.south)\tikztonodes
-- (\tikztotarget)}] &
\C \arrow[r,"\exp"]&
\C^* \arrow[r,"z\mapsto z^N"]&
\C^*.
\end{tikzcd}
\end{equation}
The map $\theta_N$ allows us to identify $U^f$ with $(U_N)^{f_N}$ in a canonical way, which we will do from now on. In particular, we can also identify the constant sheaves $\ul k_{(U_N)^{f_N}}$ and $\ul k_{U^f}$ canonically.

Let $R(N)\coloneqq k[t^{N},t^{-N}]$. Since the deck group of the infinite cyclic cover $\pi_N$ is generated by $t^N$, the corresponding Alexander modules $H_i((U_N)^{f_N};k)$ are finitely generated $R(N)$-modules. Since $R$ is a rank $N$ free $R(N)$-module, we can also consider $\calL$ as a local system of rank $N$ free $R(N)$-modules on $U$. Moreover, $\theta_N$ induces an isomorphism $\theta_{ \ovLn }\colon  p_* \ovLn \cong \ov\cL$ of local systems of $R(N)$-modules, which can be further used to prove the following (cf. \cite[Proposition 2.6.3]{EGHMW}).
\begin{lemma}\label{233}
\label{lemLocal}
In the above notations, $\theta_N$ induces the following canonical isomorphisms of $R(N)$-modules:
\[\Tors_R H^{i}(U;\ov\cL) \cong \Tors_{R(N)}H^{i}(U_N; \ovLn ) \]
for any integer $i \geq 0$. Note that $t^N$ acts unipotently on $\Tors_{R(N)}H^{i}(U_N; \ovLn )$.
\end{lemma}
\subsection{Step 2: Isolating the torsion}\label{s:isotor}
\begin{remk}
In \cite{EGHMW,EGHMWSurvey} it was said that, without loss of generality, we had reduced the problem to endowing $\Tors_R H^*(U;\ov\cL)$ with a canonical MHS in the case where $t$ acts unipotently on $\Tors_R H^*(U;\ov\cL)$ (using Lemma~\ref{233}). This simplifies the notation. However, the role that the suitable finite cover $U_N$ plays in this reduction is key for the proof of Theorem~\ref{thm:tss}, so we recall the construction of the MHS on $\Tors_{R(N)} H^*(U_N; \ovLn )$ explicitly.
\end{remk}

Let $s_N=t^N-1$. For $m \in \mathbb{N}$, we set
\[ R(N)_m\coloneqq R(N)/(s_N)^mR(N),\quad \ov{\calL_N}\coloneqq \calL_N \otimes_{t^N\mapsto t^{-N}} R(N),
\]
and let $ (\ov{\calL_N})_m\coloneqq \ov{\calL_N} \otimes_{R(N)} R(N)_m $ be the corresponding rank one local system of $R(N)_m$-modules.
Since $s_N$ acts nilpotenly on $\Tors_{R(N)} H^*(U_N;\ov{\calL_N})$, there exists an integer $m\geq 1$ such that $(s_N)^m$ annihilates $\Tors_{R(N)} H^j(U_N;\ov{\calL_N})$, for all $j$. With the above notations, it can be seen (cf. \cite[Lemma 3.1.8,  Corollary 3.1.9]{EGHMW}) that the maps of sheaves
\[
\ov{\calL_N} \twoheadrightarrow (\ov{\calL_N})_m \xhookrightarrow{\cdot (s_N)^m} (\ov{\calL_N})_{2m}
\]
induce an exact sequence
\[
0 \to \Tors_{R(N)} H^*(U_N;\ov{\calL_N})  \to H^*(U_N;(\ov{\calL_N})_m) \xrightarrow{\cdot (s_N)^m}  H^*(U_N;(\ov{\calL_N})_{2m}).
\]
Hence, since $s_N$ and $\log(1+s_N)$ differ by multiplication by a unit in $R(N)_{2m}$,
\begin{equation}\label{ker} 
\Tors_{R(N)} H^*(U_N;\ov{\calL_N})  \cong \ker \left( H^*(U_N;(\ov{\calL_N})_m) \xrightarrow{\cdot \left(\log(1+s_N)\right)^m}  H^*(U;{(\ov{\calL_N})}_{2m}) \right)  .
\end{equation}

Our next goal is to endow each $H^i(U_N;(\ov{\calL_N})_m)$ with a canonical mixed Hodge structure for all $m \geq 1$, such that the map 
\[ H^*(U_N;(\ov{\calL_N})_m) \xrightarrow{\cdot \left(\log(1+s_N)\right)^m} H^*(U_N;(\ov{\calL_N})_{2m})(-m) \]
 is a morphism of mixed Hodge structures (where $(-m)$ denotes the $-m$th Tate twist). This will be achieved by resolving $(\ov{\calL_N})_m$ by a certain mixed Hodge complex. Note that such a mixed Hodge structure on $H^*(U_N;(\ov{\calL_N})_m)$  induces by equation \eqref{ker} a canonical mixed Hodge structure on $\Tors_{R(N)} H^*(U_N;\ov{\calL_N})$ (independent of $m\gg 0$ by \cite[Corollary 5.4.5]{EGHMW}), which by Lemma~\ref{lemLocal} induces a canonical mixed Hodge structure on $\Tors_{R} H^*(U;\ov{\calL})$ (which, by \cite[Theorem 5.4.8]{EGHMW}, does not depend on the choice of suitable $N$).

\subsection{Step 3: Resolving \texorpdfstring{$(\ov{\calL_N})_m$}{L N m} by a mixed Hodge complex}\label{s:thickldr}

Deligne used a mixed Hodge complex of sheaves to endow the cohomology of any smooth algebraic variety $V$ with a canonical mixed Hodge structure \cite{De2}. See \cite[II.4.11]{peters2008mixed} (or \cite[Section 2.9]{EGHMW} for the version over $\R$ used in this paper) for the precise definition of this mixed Hodge complex of sheaves.

\begin{remk}\label{remk:MHComplex}
We are not going to recall the definition of mixed Hodge complex of sheaves here. For our purposes, we only need to know that 
\begin{enumerate}
\item an $\R$-mixed Hodge complex of sheaves consists on several pieces of data, some of which are a complex of sheaves of real vector spaces (which we will call the ``real part'') endowed with an increasing (weight) filtration, and a complex of sheaves of $\C$-vector spaces (which we will call the ``complex part'') endowed with both an increasing (weight) and a decreasing (Hodge) filtration. 
\item $\R$-mixed Hodge complexes of sheaves endow the hypercohomology of their real part (a complex of sheaves of $\R$-vector spaces) with an $\R$-mixed Hodge structure.
\item morphisms of mixed Hodge complexes of sheaves induce morphisms of MHS in hypercohomology.
\end{enumerate}
\end{remk}

In Deligne's $\R$-mixed Hodge complex of sheaves, the real part is $j_*\cE^\bullet_V$, where $\cE^\bullet_V$ is the de Rham complex of real differential forms, and $j:V\hookrightarrow X$ is the inclusion of $V$ into a \textit{good compactification} $X$, that is, a smooth compactification such that $X\setminus V$ is a normal crossings divisor.

In particular, $j_*\cE^\bullet_V$ is canonically identified with $Rj_*\underline{\R}_V$, so Deligne's mixed Hodge complex of sheaves endows $H^*(V;\R)\cong \mathbb{H}^*(V, j_*\cE^\bullet_V)$ with a canonical MHS.

In \cite{EGHMW}, $( \ovLn )_m$ (or rather $Rj_*( \ovLn )_m$, where $j$ is an inclusion of $U_N$ into a good compactification) was resolved by a complex of sheaves which is the real part of a mixed Hodge complex of sheaves, endowing $H^*(U_N;( \ovLn )_m)$ with a canonical MHS. We now recall that construction.

\begin{dfn}\label{dfn:thickening}
Let $(\calA,\wedge,d)=\ldots\xrightarrow{d}\calA^i \xrightarrow{d}\calA^{i+1}\xrightarrow{d}\ldots$ be a bounded below sheaf of commutative differential graded algebras (cdgas) on a topological space $X$.
For any $m \geq 1$, the {\it $m$-thickening} of a sheaf of $k$-cdgas $(\calA,\wedge,d)$ in the direction $\eta \in \Gamma(X, \calA^1) \cap \ker d$ is the cochain complex of $R{(N)}_m$-modules denoted by 
\begin{align*}
	\calA(\eta,m) = (\calA {\otimes_k} R(N)_m, d_\eta)
\end{align*}
 and described by:
\begin{itemize}
\item[(i)] for $p \in \Z$, the $p$-th graded component of $\calA(\eta,m)$ is {$\calA^p {\otimes_k} R(N)_m$.}
\item[(ii)] for $\omega \in \calA^i$ and $\phi \in R(N)_m$, we set $d_\eta(\omega \otimes \phi) = d\omega \otimes \phi + (\eta \wedge \omega) \otimes {s_N\phi} $.
\end{itemize}
\end{dfn}

Let $\ker d^0$ be the kernel of the $0$-th differential of $\calE^\bullet_{U_N}(\Im \frac{df_N}{f_N},m)$, where $\Im$ denotes the imaginary part. By \cite[Lemma 5.2.5]{EGHMW}, $\ker d^0$ is canonically isomorphic to $(\ov{\calL_N})_m$, as sheaves of $\R$-vector spaces. The isomorphism is given by the map
$$
\nu\colon (\ov{\calL_N})_m=\ov\calL_N\otimes_{R(N)} R(N)_m\xrightarrow{\cong} \ker\, d^0_m
$$
defined as follows (see \cite[Remark 5.2.6]{EGHMW} for details).

On any simply connected neighborhood $V$ of $x\in U_N$, $ \ovLn $ has the sections of $\pi_N\colon (U_N)^{f_N}\to U_N$ as a basis. Let $\iota\colon V\to (U_N)^{f_N}$ be one of these sections. Then $\nu$ is defined by the following formula, and extended $\R$-linearly:
\[
\nu (\iota\otimes 1) =\exp(-\Im (f_N)_\infty\circ \iota\otimes s_N).
\]
$\nu$ defines an isomorphism of sheaves. It does not respect the $R(N)_m$-module structure, but it satisfies equations
\begin{equation}\label{eqn:nulinear}
\begin{split}
\nu \left(a\cdot \frac{\log(1+s_N)}{2\pi}\right)&=\nu (a)\cdot s_N\\
\nu \left(a\cdot s_N\right)&=\nu (a)\cdot (\exp(2\pi s_N)-1)
\end{split}
\end{equation}
for all local sections $a\in (\ov\calL_N)_m$, where, $\log(1+s_N)$ and $\exp(2\pi s_N)$ are seen as  elements in $R(N)_m$ as truncated power series on $s_N$.

We can perform thickenings of mixed Hodge complexes of sheaves to obtain other mixed Hodge complexes of sheaves through the theory developed in \cite[Section 4]{EGHMW}. In particular, we can thicken Deligne's mixed Hodge complex of sheaves (see \cite[Section 5.4]{EGHMW}) to obtain a mixed Hodge complex of sheaves whose real part is $j_*\cE_{U_N}^\bullet(\Im \frac{df_N}{f_N},m)$. This endows $
H^*(U_N;(\ov{\calL_N})_m)
\cong
 \mathbb H^*\left(U_N,\cE_{U_N}^\bullet(\Im \frac{df_N}{f_N},m)\right)
 \cong
  \mathbb{H}^*\left(X,j_*\cE_{U_N}^\bullet(\Im \frac{df_N}{f_N},m)\right)$ with a canonical MHS, where the first isomorphism is induced by $\nu$.

\subsection{Step 4: Constructing the MHS on \texorpdfstring{$\Tors_R H^*(U;\ov\cL)$}{Tors R H*(U,L)}}\label{s:MHS}

In Section~\ref{s:thickldr}, we have constructed a canonical $\R$-MHS on $H^i(U_N;(\ov{\calL_N})_m)$. In \cite[Theorem 5.4.7]{EGHMW}, it is shown that this MHS does not depend on the choices made in its construction, namely the good compactification.

In Section~\ref{s:isotor}, we recalled that the maps of sheaves $\ov{\calL_N}\twoheadrightarrow (\ov{\calL_N})_m \xrightarrow{\cdot (s_N)^m} (\ov{\calL_N})_{2m}$ induce the exact sequence in cohomology
\begin{equation}\label{eq:isolateTorsion}
0\rightarrow \Tors_{R(N)}H^i(U_N;\ov{\calL_N})\rightarrow H^i(U_N;(\ov{\calL_N})_m)\xrightarrow{\cdot \left(\frac{\log(1+s_N)}{2\pi}\right)^m} H^i(U_N;(\ov{\calL_N})_{2m},
\end{equation}
for big enough $m$ (in fact, we just need that $(s_N)^m=(t^N-1)^m$ annihilates $\Tors_{R(N)}H^i(U_N;\ov{\calL_N})$).

Moreover, by \cite[Lemma 4.2.3]{EGHMW} and the way $\nu$ is defined,
$$H^i(U_N;(\ov{\calL_N})_m)\xrightarrow{\cdot \left(\frac{\log(1+s_N)}{2\pi}\right)^m} H^i(U_N;(\ov{\calL_N})_{2m}(-m)$$
is a MHS morphism (realized by multiplication by $(s_N)^m$ on the corresponding thickened Hodge-de Rham complexes of sheaves), where the domain and the target are endowed with the MHS described at the end of Section~\ref{s:thickldr}, and $(-m)$ denotes the $-m$th Tate twist. Hence, the natural projection $\ov{\calL_N}\twoheadrightarrow (\ov{\calL_N})_m$ endows $\Tors_{R(N)}H^i(U_N;\ov{\calL_N})$ with a canonical MHS.

Lastly, Lemma~\ref{lemLocal} provides the isomorphism induced by $\theta_N$ between $\Tors_{R}H^i(U;\ov{\cL})$ and $\Tors_{R(N)}H^i(U_N;\ov{\calL_N})$, which we can use to endow $\Tors_{R}H^i(U;\ov{\cL})$ with a canonical MHS.

\begin{remk}\label{remk:Q}
Note that Alexander modules are defined over $\Q$. In \cite[Theorem 5.4.10]{EGHMW}, it was proved that with this $\Q$-structure, $\Tors_R H^*(U;\ov{\cL})$ is a $\Q$-MHS. Remark~\ref{remk:QRC} implies that proving Theorem~\ref{thm:tss} over $\R$ implies the analogous statement over $\Q$ and $\C$.
\end{remk}

\begin{remk}\label{remk:noN}
Let $s=t-1$, and let $m\gg 0$. Let $\left(\Tors_{R} H^*(U;\ov{\calL})\right)_1$ be the generalized eigenspace of eigenvalue $1$. We can carry out the same construction from Sections~\ref{s:isotor} and~\ref{s:thickldr}, but on $U$ instead of $U_N$, replacing the ring $R(N)$ for $R$ and $s_N$ by $s$. We conclude by the same arguments that
$$
\left(\Tors_{R} H^*(U;\ov{\calL})\right)_1\cong \ker\left(H^*(U;\ov{\calL}_m)\xrightarrow{\cdot \left(\frac{\log(1+s)}{2\pi}\right)^m}H^*(U;\ov{\calL}_{2m})\right),
$$
which, following the arguments in this section but without passing to a finite cover, yields
$$
\left(\Tors_{R} H^*(U;\ov{\calL})\right)_1\cong \ker\left(\mathbb{H}^*\left(U,\cE^\bullet_U\left(\Im\frac{df}{f},m\right)\right)\xrightarrow{\cdot s^m}\mathbb{H}^*\left(U,\cE^\bullet_U\left(\Im\frac{df}{f},2m\right)\right)(-m)\right).
$$
Multiplication by $s^m$ above is a MHS morphism, so its kernel $\left(\Tors_{R} H^*(U;\ov{\calL})\right)_1$ is endowed with a MHS. In principle, this MHS on $\left(\Tors_{R} H^*(U;\ov{\calL})\right)_1$ is not necessarily the MHS that it has as a subspace of $\Tors_{R} H^*(U;\ov{\calL})$, whose construction we just recalled (it used a finite cover $U_N$). However, in Corollary~\ref{cor:noCoverNeeded}, we prove that those two MHS coincide. Hence, as anticipated in the introduction, the MHS on $\left(\Tors_{R} H^*(U;\ov{\calL})\right)_1$ can be computed without the knowledge of a particular suitable value of $N$ (or its corresponding finite cover $U_N$), whose existence is guaranteed by Theorem~\ref{thm:summary}, part~\eqref{part:N}.
\end{remk}
\subsection{The functoriality of the MHS on \texorpdfstring{$\Tors_R H^*(U;\ov\cL)$}{Tors R H*(U,L)}}
Let $g:U_1\rightarrow U_2$ be an algebraic map between smooth algebraic varieties, and let $f_i:U_i\rightarrow \C^*$ be an algebraic map inducing a surjection in fundamental groups for $i=1,2$. Suppose further that $f_2\circ g=f_1$. Let $\cL^i$ be the rank $1$ local system of free $R$-modules on $U_i$ induced by $f_i$. Note that $\cL^1=g^{-1}\cL^2$, so $\ov{\cL^1}=g^{-1}\ov{\cL^2}$. The map $\ov{\cL^2}\rightarrow Rg_*\ov{\cL^1}$ induced by the adjunction $\id\to Rg_*g^{-1}$ induces a map $\Tors_R H^*(U_2;\ov{\cL^2})\rightarrow \Tors_R H^* (U_1;\ov{\cL^1})$ which is the dual of the map in homology coming from $U_1^{g_1}\to U_2^{g_2}$ (\cite[Proposition 2.4.1]{EGHMW}).

Let $N\in\N$ such that the action of $t^N$ on $\Tors_R H^*(U_i;\ov{\calL^i})$ is unipotent, for $i=1,2$. The map $g$ lifts to a map $g_N:(U_1)_N\rightarrow (U_2)_N$ such that $\ov{\calL^1_N}=g_N^{-1}\ov{\calL^2_N}$. As explained in \cite[Section 4.5.1]{peters2008mixed}, it is possible to find smooth compactifications $X_i$ of $(U_i)_N$ such that $X_i\backslash (U_i)_N$ is a simple normal crossing divisor for $i=1,2$ and such that $g_N$ extends to $g_N: X_1\rightarrow X_2$. Let $j_i:(U_i)_N\hookrightarrow X_i$ be the inclusion for $i=1,2$.

Under these assumptions and notation, the content of \cite[Theorem 5.4.9]{EGHMW} is as follows:
\begin{thm}\label{thm:functorial}
The map $\Tors_R H^*(U_2;\ov{\cL^2})\rightarrow \Tors_R H^* (U_1;\ov{\cL^1})$ induced by adjunction is a morphism of MHS, realized by a morphism of (thickened) mixed Hodge complexes of sheaves for big enough $m$. The map between the real parts of these mixed Hodge complexes of sheaves is given by $R(j_2)_*$ of the map

\begin{equation}\label{eq:pullbackMap1}
\calE_{(U_2)_N}^{\bullet}\left(\Im\frac{d(f_2)_N}{(f_2)_N},m\right)\rightarrow (g_N)_*\left(\calE_{(U_1)_N}^{\bullet}\left(\Im\frac{d(f_1)_N}{(f_1)_N},m\right)\right),
\end{equation}
defined at each degree $i$ by the $R(N)_m$-linear map
\begin{equation}\label{eq:pullbackMap2}
\begin{array}{ccc}
\calE_{(U_2)_N}^i\otimes R(N)_m &\rightarrow &(g_N)_*\left(\calE_{(U_1)_N}^{i}\otimes R(N)_m\right)\\
\omega\otimes 1&\mapsto& (g_N)^*\omega\otimes 1.
\end{array}
\end{equation}
\end{thm}

\begin{remk}\label{remk:push}
Suppose that, with the above hypotheses, instead of having the condition $g^* f_2=f_1$ in  we have the weaker condition $f_2\circ g = \lambda\cdot f_1$ for some $\lambda\in \C^*$. In this case, $g^* \frac{df_2}{f_2}=\frac{df_1}{f_1}$, and the map described in {equations} \eqref{eq:pullbackMap1} and \eqref{eq:pullbackMap2} is still a morphism of mixed Hodge complexes, simply by applying Theorem~\ref{thm:functorial} to the maps $f_2$ and $\lambda f_1$.

This morphism between $m$-thickened de Rham complexes commutes with multiplication by $(s_N)^m$ into their $2m$-th thickenings (which allows us to isolate the torsion using {equation} \eqref{eq:isolateTorsion}), and thus it induces a MHS morphism between $\Tors_R H^*(U_2;\ov{\cL^2})$ and $\Tors_R H^*(U_1;\ov{\cL^1})$. The only difference with Theorem~\ref{thm:functorial} is that, in this case $g^{-1}\cL^2$ is non-canonically isomorphic to $\cL^1$, so the map between the torsion parts of the Alexander modules will not be the one induced by adjunction in general.
\end{remk}

\section{The semisimple part of the monodromy is a MHS morphism}\label{s:semisimple}

Let $N$ such that $t^N-1$ acts nilpotently on $\Tors_R H^{i}(U;\ov\cL)$ for all $i$. Recall the diagrams and notation from Section~\ref{s:unip}. We will often represent $\cL$ as the sheaf such that on every open set $V$, $\Gamma(V,\cL)$ is the $\R$-vector space with basis given by the set of sections of $\pi$ defined over $V$. Similarly, $\cL_N$ has the sections of $\pi_N$ as a basis.

Let $\cT$ generate the deck transformations of $\pi$, i.e. for $(x,z)\in U^f\subset U\times \C$, $\cT(x,z) = (x,z+2\pi i)$. By definition, $t\in R$ has an action on $\cL$ induced by $\cT$, i.e. it takes the local section $\iota$ such that $\iota(x) = (x,z)$ and maps it to $t\cdot\iota$, such that $(t\cdot\iota)(x)= (x,z+2\pi i)$. In other words, $t\cdot\iota = \cT\circ \iota$. Analogously, $t^N$ acts on sections of $\pi_N$ by composing with $\cT_N\coloneqq \theta_N\circ \cT^N \circ \theta_N^{-1}$, which is the generator of deck transformations of $\pi_N$.

$\cT$ descends to the deck transformation of $p$, which we will also denote $\cT\colon U_N\to U_N$. Also, we can use $\theta_N$ to define $\cT_N^{1/N}\coloneqq \theta_N\circ \cT \circ \theta_N^{-1}$, which is indeed a canonically chosen $N$-th root of $\cT_N$, and further it descends to $\cT$ acting on $U_N$.

$\cL$ and $\ov\cL$ are canonically identified, but the action of $t$ on $\ov\cL$ equals the action of $t^{-1}$ on $\cL$. That is, if $\iota\colon U\to U^f$ is seen as a local section of $\ov\cL$, $t\cdot \iota = \cT^{-1}\circ \iota$, and analogously for $ \ovLn $.

Let us recall the isomorphism $\theta_{ \ovLn }\colon p_* \ovLn \xrightarrow{\sim} \ov\cL$ from \cite[Lemma 2.6.3]{EGHMW}, which is responsible for the result of Lemma~\ref{lemLocal}.

For any open set $V\subseteq U$, we have an isomorphism 
$
\Gamma(p^{-1}(V); \ovLn ) \to \Gamma(V;p_* \ovLn )$ that we will denote $p_*$.

Let $x\in U$, and let $V$ be a simply connected neighborhood of $x$. Note that $\Gamma(p^{-1}(V);\ov{\cL_N})$ decomposes as a direct sum
$\bigoplus_{x'\in p^{-1}(x)}\Gamma(V_{x'};\ov{\cL_N})$
, where
$V_{x'}$ is a neighborhood of $x'$ and\linebreak 
$p:V_{x'}\rightarrow V$ 
is a homeomorphism for all $x'\in p^{-1}(x)$. Let $\wt x\in p^{-1}(x)$, and let 
$\iota\in \Gamma(V_{\wt x};\ov{\cL_N})$ 
be a local section of $\pi_N$ around $\wt x$. Seeing 
$\iota$ inside $\Gamma(p^{-1}(V);\ov{\cL_N})$
, we have that 
$p_*\iota\in \Gamma(V;p_*\ov{\cL_N})$. 
Note that all such $p_*\iota$ form a basis of $\Gamma(V;p_*\ov{\cL_N})$. Let $\iota_N$ be the local section of $p$ mapping $x$ to $\wt x$. Then $\theta_N^{-1}\circ \iota\circ \iota_N$ is a local section of $\pi$ around $x$, so we can define:

\[
\theta_{ \ovLn }(p_*\iota) = \theta_N^{-1}\circ \iota\circ \iota_N.
\]

\begin{lemma}\label{lem:t-liftstoEquiv}
There is a morphism of sheaves $\cM_t\colon  \ovLn  \to \cT_* \ovLn $ such that after taking $p_*$ it becomes multiplication by $t$, i.e. the following composition is multiplication by $t$:
\[
\ov\cL \xrightarrow[\sim]{\theta_{ \ovLn }^{-1}} p_* \ovLn  \xrightarrow{p_* \cM_t} p_* \cT_* \ovLn    = (p\circ \cT)_*  \ovLn  = p_* \ovLn  \xrightarrow[\sim]{\theta_{ \ovLn }} \ov\cL.
\]
$\cM_t$ is given by $\cM_t(\iota)= \cT_N^{-1/N} \circ \iota\circ \cT$.
\end{lemma}
\begin{proof}
We want to see that the following diagram commutes (note that the vertical arrows are isomorphisms).
\[
\begin{tikzcd}[column sep = 7em, row sep = 2em]
p_* \ovLn  \arrow[r,dashrightarrow,"p_*\cM_t"]\arrow[d,"\theta_{ \ovLn }"]& p_* \ovLn \arrow[d,"\theta_{ \ovLn }"] 
\\
\ov\cL\arrow[r,"t"] & \ov\cL .
\end{tikzcd}
\]
Let $\iota\colon U_N\to (U_N)^{f_N}$ be a local section of $ \ovLn $ around $\wt x\in U_N$. Let $x=p(\wt x)$. We would like to check that $\theta_{ \ovLn }( p_*(\cM_t\iota)) = t\theta_{ \ovLn }(p_*\iota)$. Let $\iota_N$ be the local section of $p$ mapping $x$ to $\wt x$. Let $\iota_N'$ be the local section of $p$ mapping $x$ to $\cT^{-1}\wt x$. Note that $\iota_N'=\cT^{-1}\circ \iota_N$.

\begin{align*}
\theta_{ \ovLn }(p_*(\cM_t\iota)) 
&= 
\theta_{ \ovLn }(p_*(\cT_N^{-1/N}\circ \iota \circ \cT)) =&\text{(def. of $\cM_t$)}\\
&= 
\theta^{-1}_N\circ \cT_N^{-1/N}\circ \iota \circ \cT \circ \iota_N' =&\text{(def. of $\theta_{ \ovLn }$ at $\cT^{-1}\wt x$)}\\ 
&= \cT^{-1}\circ\theta^{-1}_N\circ  \iota \circ \cT \circ \iota_N'= & \text{(}\theta_N\circ \cT\circ \theta^{-1}_N= \cT_N^{1/N}\text{)}
\\ 
&= \cT^{-1}\circ\theta^{-1}_N\circ  \iota \circ  \iota_N= & \text{(}\iota_N =  \cT \circ \iota_N'\text{)}
\\
&= t\cdot  (\theta^{-1}_N\circ \iota \circ \iota_N)=
&\text{(def. of $t\cdot$)}
\\
&=
t\theta_{ \ovLn }(p_*\iota).
 &\text{(def. of $\theta_{ \ovLn }$ at $\wt x$)}\
\end{align*}

\end{proof}

\begin{remk}
We are going to look at the map that $\cM_t$ induces between $( \ovLn )_m$ and itself, namely $\cM_t\otimes_{R(N)}R(N)_m$. For simplicity in the notation, we will denote $\cM_t\otimes_{R(N)}R(N)_m$ also by $\cM_t$ from now on. 
\end{remk}

\begin{lemma}\label{lem:t-liftstodR}
Let $k=\R$. Under the quasi-isomorphism $\nu \colon  \ovLn \otimes_{R(N)}{R(N)_m} \to \cE^\bullet_{U_N}\left( \Im \frac{df_N}{f_N},m\right)$, the map $\cM_t$ (where $\cM_t$ was defined in Lemma~\ref{lem:t-liftstoEquiv}) becomes the following morphism of complexes of sheaves, defined $R(N)$-linearly as
\[
\f{\wt\cM_t}{\cE^\bullet_{U_N}\left( \Im \frac{df_N}{f_N},m\right)}{\cT_*\cE^\bullet_{U_N}\left( \Im \frac{df_N}{f_N},m\right)}{\omega\otimes 1}{\cT^*\omega\otimes e^{2\pi s_N/N}}
\]
In other words, the following diagram commutes:
\[
\begin{tikzcd}
 \ovLn \otimes_{R(N)}{R(N)_m} \arrow[r,"\cM_t"]\arrow[d,"\nu"] &
\cT_* \ovLn \otimes_{R(N)}{R(N)_m}\arrow[d,"\cT_*\nu"] \\
\cE^\bullet_{U_N}\left( \Im \frac{df_N}{f_N},m\right) \arrow[r,"\wt\cM_t"] &
\cT_*\cE^\bullet_{U_N}\left( \Im \frac{df_N}{f_N},m\right)
\end{tikzcd}
\]
\end{lemma}
\begin{proof}
The map $\nu$ from Section~\ref{s:thickldr} was defined by mapping a local section $\iota\colon U_N\to (U_N)^{f_N}$ to
$\exp(-\Im (f_N)_\infty\circ \iota \otimes s_N)$. We need to show that $\wt\cM_t \circ \nu = \cT^*\nu\circ \cM_t$. We can verify this on a basis made of sections $\iota$ of $\pi_N$.

\begin{align*}
\nu (\cM_t\iota) &= \nu((\cT_N^{-1/N}\circ \iota\circ \cT)\otimes 1) =  &\text{(def. of $\cM_t$)}\\
&= \exp(-\Im (f_N)_\infty\circ \cT_N^{-1/N}\circ \iota\circ \cT\otimes s_N) = &\text{(def. of $\nu $)}\\
&= \exp(-\Im ((f_N)_\infty\circ \iota\circ \cT-2\pi i /N)\otimes s_N) = &\text{(def. of $\cT_N^{1/N}$)}\\
&= \exp(2\pi s_N/N) \exp(-\Im (f_N)_\infty\circ \iota\circ \cT\otimes s_N) = \\
&= \exp(2\pi s_N/N) \exp(-\cT^*(\Im (f_N)_\infty\circ \iota)\otimes s_N)=
\\
&=
 \wt\cM_t\exp(-\Im (f_N)_\infty\circ \iota \otimes s_N) 
 & \text{(def. of $\wt\cM_t$)}
 \\
&=
\wt\cM_t \circ \nu(\iota). 
 & \text{(def. of $\nu$)}
 \\
\end{align*}

\end{proof}

By Lemmas~\ref{lem:t-liftstoEquiv} and~\ref{lem:t-liftstodR} together, the following diagram commutes. Note that $p_*\circ \cT_* = p_*$.
\[
\begin{tikzcd}
\ov\cL \arrow[r,"\theta_{ \ovLn }^{-1}","\sim"']\arrow[d,"t\cdot"]&
p_* \ovLn  \arrow[r,twoheadrightarrow]\arrow[d,"p_*\cM_t"] &
p_*( \ovLn \otimes R(N)_m)\arrow[r,"p_*\nu","\sim"']\arrow[d,"p_*\cM_t"] &
p_*\cE^\bullet_{U_N}\left( \Im \frac{df_N}{f_N},m\right)\arrow[d,"p_*\wt\cM_t"]\\
\ov\cL \arrow[r,"\theta_{ \ovLn }^{-1}","\sim"']&
p_* \ovLn  \arrow[r,twoheadrightarrow] &
p_*( \ovLn \otimes R(N)_m)\arrow[r,"p_*\nu","\sim"'] &
p_*\cE^\bullet_{U_N}\left( \Im \frac{df_N}{f_N},m\right).
\end{tikzcd}
\]
Taking hypercohomology of the above diagram, we have that ${p_*}\wt \cM_t$ induces an endomorphism of $\H^*\left(U;p_*\cE^\bullet_{U_N}\left( \Im \frac{df_N}{f_N},m\right)\right) \cong  \H^*\left(U_N;\cE^\bullet_{U_N}\left( \Im \frac{df_N}{f_N},m\right)\right)$, which in turn induces the action of $t$ on $\Tors_R H^*(U;\ov\cL)$ if $m\gg 0$. Indeed, if $m\gg 0$, the horizontal maps are the ones used in Section~\ref{s:MHS} to canonically embed $\Tors_R H^*(U;\ov\cL)$ into $\H^*(U_N;(\ov{\calL_N})_m)\cong\H^*\left(U_N;\cE^\bullet_{U_N}\left( \Im \frac{df_N}{f_N},m\right)\right)$, where the isomorphism is given by $\nu$.

 In the next proposition, which completes the proof of Theorem~\ref{thm:tss}, we find a morphism $\wt\cM_t^{ss}$ that induces the semisimple part of $\wt \cM_t(=t)$ on $\H^*\left(U_N;\cE^\bullet_{U_N}\left( \Im \frac{df_N}{f_N},m\right)\right)$. This is meant in the sense that $t$ decomposes as a product of two commuting operators $t=t_{ss} t_{u}$, where $t_{ss} $ is semisimple and $t_{u}$ is unipotent.

We are ready to prove Theorem~\ref{thm:tss}. In fact, we prove the following more precise statement with $\R$ coefficients. The result for $\Q$ and $\C$ follows from Remark~\ref{remk:QRC} and Remark~\ref{remk:Q}.
\begin{thm}\label{prop:tss}
Let $t=t_{ss}t_{u}$ be the Jordan-Chevalley decomposition of the action of $t$ on $\Tors_R H^i(U;\ov\cL)$ as the product of a semisimple operator and a unipotent operator that commute with each other.  The action of $t_{ss}  $ on $\Tors_R H^i(U;\ov\cL)$ is a mixed Hodge structure morphism induced by the following morphism of complexes of sheaves for $m\gg 0$, defined $R(N)$-linearly as
\[
\f{\wt\cM_t^{ss}}{\cE^\bullet_{U_N}\left( \Im \frac{df_N}{f_N},m\right)}{\cT_*\cE^\bullet_{U_N}\left( \Im \frac{df_N}{f_N},m\right)}{\omega\otimes 1}{\cT^*\omega\otimes 1.}
\]
\end{thm}
\begin{proof}
Notice that $f_N\circ \cT= e^{2\pi i / N} f_N$, so $\cT^*\frac{df_N}{f_N}=\frac{df_N}{f_N}$. By Remark~\ref{remk:push}, $\wt\cM_t^{ss}$ is a well defined chain map that extends to a morphism of mixed Hodge complexes of sheaves, and thus induces a MHS morphism from $H^*(U_N;( \ovLn )_m)$ to itself, which it turn induces a MHS morphism from $\Tors_R H^*(U;\ov\cL)$ to itself. Hence, all that is left to show is that $\wt\cM_t^{ss}$ induces $t_{ss} $ on $\Tors_R H^*(U;\ov\cL)$.

Recall that $\wt\cM_t (\omega\otimes f(s_N)) = \cT^*\omega \otimes e^{2\pi s_N/N} f(s_N)$. Since we are interested in the corresponding map in cohomology, we can consider $p_*\wt\cM_t^{ss}$ instead of $\wt \cM_t^{ss}$. It satisfies the following properties:
\begin{enumerate}
\item $p_*\wt\cM_t$ and $p_*\wt\cM_t^{ss}$ commute: this can be checked directly.
\item $(p_*\wt\cM_t^{ss})^N$ is the identity.
\item $p_*\wt\cM_t \circ (p_*\wt\cM_t^{ss})^{-1}$ is unipotent: By the definitions, it is multiplication by $e^{2\pi s_N/N}$, which is indeed unipotent, as $e^{2\pi s_N/N}-1$ is a multiple of $s_N$, and $(s_N)^m=0$ in $R(N)_m$
\end{enumerate}
Let $A$ and $B$ be the maps induced on $\Tors_R H^*(U;\ov\cL)$ by $p_*\wt\cM_t^{ss}$ and $p_*\wt\cM_t \circ (p_*\wt\cM_t^{ss})^{-1}$ respectively. Recall that $p_*\wt\cM_t$ induces multiplication by $t$ on $\Tors_R H^*(U;\ov\cL)$ by Lemmas~\ref{lem:t-liftstoEquiv} and \ref{lem:t-liftstodR}, so $t=AB$. Property (1) above says that $A$ and $B$ commute, property (2) above implies that $A$ is semisimple, and property (3) says that $B$ is unipotent. By the uniqueness of the Jordan-Chevalley decomposition, $A=t_{ss} $ and $B=t_{u}$, which concludes the proof.
\end{proof}

\begin{remk}\label{remk:decompcohom}
Notice that Theorem~\ref{thm:tss} is the dual statement (with a shift, see Remark~\ref{remk:dual}) of Theorem~\ref{prop:tss}, and that the direct sum decomposition from Corollary~\ref{cor:geneigenspaces} also holds for $\Tors_R H^*(U;\ov\cL)$.
\end{remk}

To understand the direct sum decomposition of Corollary~\ref{cor:geneigenspaces}, it suffices to work in cohomology, as exemplified by the following result

\begin{corollary}\label{cor:dual}
Let $g\in k[x]$ be a factor of $x^N-1$, with $k=\Q,\R,\C$. Under the notation of Corollary~\ref{cor:geneigenspaces}, there is a MHS isomorphism $(\Tors_R H^{j+1}(U,\ov\cL))_{g}\cong \left((\Tors_R H_{j}(U^f,k))_{g}\right)^{\vee_k}$ for all $j$. 
\end{corollary}
\begin{proof}
Since $\Tors_R H^{j+1}(U,\ov\cL)$ and $\Tors_R H_{j}(U^f,k))$ are dual MHS's, each sub-MHS of the former is dual to a quotient MHS of the latter. Note that the dual of multiplication by $t$ is multiplication by $t$, so the $k$-dual of $t_{ss} $ is $t_{ss}$. Combining these two facts, we have that the kernel of $g(t_{ss})$ (which is a sub-MHS by Theorem~\ref{thm:tss}) is naturally the dual MHS to the cokernel of $g(t_{ss})$ acting on $\Tors_R H_{j}(U^f,k))$. Since both decompose as a direct sum of their eigenspaces for $t_{ss}$, there is a natural MHS isomorphism between the cokernel and the kernel of $g(t_{ss})$ (namely, the action of $\frac{x^N-1}{g(x)}$ evaluated at $t_{ss}$).

\end{proof}

\section{The MHS on the eigenspace of eigenvalue 1}\label{s:eigen1}

In this section, we prove the result promised in Remark~\ref{remk:noN}, namely that the MHS on $\left(\Tors_R H^*(U;\ov\cL)\right)_1$ can be computed from $U$, without using a finite cover. We use the notations from this Remark. Let us recall the constructions: in Section~\ref{s:thickldr} it is mentioned that $\calE^\bullet_{U}(\Im \frac{df}{f},m)$ has a mixed Hodge complex structure, and a quasi-isomorphism with $\ov{\calL}_m$ is chosen. This gives a MHS to $H^*(U;\ov{\calL}_m)$.

In the following proposition, we recall relevant facts which appear in the proof of \cite[Theorem 5.4.8]{EGHMW} and we draw new conclusions stemming from Theorem~\ref{thm:tss} and its proof.
\begin{proposition}\label{prop:recallingIndepN}
Let $\sigma'_m$ be the following composition multiplied by the scalar $\frac{1}{N}$:
\begin{equation}\label{eqn:sigma'}
p_* \ovLn \otimes_{R(N)}R(N)_m
\xrightarrow[\cong]{\theta_{ \ovLn }\otimes_{R(N)} R(N)_m}
\ov{\cL}\otimes_{R(N)}R(N)_m
\overset{a}\twoheadrightarrow 
\ov{\cL}\otimes_{R}R_m,
\end{equation}
where the map $a$ is induced by tensoring $\id_{\ov\cL}$ with the inclusion $R(N)\hookrightarrow R$. Then $\sigma'_m$ is a split surjection and it has a right inverse $\sigma_m$ with the following properties.
\begin{enumerate}

\item \label{part:sigmaMHC}$\sigma_m$ is the real part of a morphism of mixed Hodge complexes. In particular, it induces a split injective MHS morphism:
\[
H^*(U;\ov{\cL}\otimes_R R_m)\hookrightarrow H^*(U_N; \ovLn \otimes_{R(N)}R(N)_m).
\]

\item \label{part:Rlinear}$\sigma_m$ and $\sigma'_m$ are morphisms of sheaves of $R$-modules, where $p_* \ovLn \otimes_{R(N)}R(N)_m$ has the $R$-module structure given by letting $t$ act as $p_*\cM_t:p_* \ovLn \rightarrow p_* \ovLn $. In particular, they are $R(N)$-linear.

\item \label{part:phipsi}Let $\phi_{m'm}$ be the projection $\ov{\cL}\otimes_R R_{m+m'}\twoheadrightarrow \ov{\cL}\otimes_R R_{m}$, and let $\psi_{mm'}:\ov{\cL}\otimes_R R_m\hookrightarrow \ov{\cL}\otimes_R R_{m+m'}$ be given by multiplication by $(\log(1+s))^{m'}$ (seen as a power series in $s=t-1$). Similarly, we can define $(\phi_N)_{m'm}$ and $(\psi_N)_{mm'}$ using $s_N=t^N-1$. The following hold up to multiplication by a non-zero real constant:
\begin{align*}
\sigma_{m+m'}\circ\psi_{mm'}&=(\psi_N)_{mm'}\circ\sigma_{m};
&
\sigma'_{m+m'}\circ(\psi_N)_{mm'}&=\psi_{mm'}\circ\sigma'_{m};
\\
\sigma_m\circ\phi_{m'm}&=(\phi_N)_{m'm}\circ\sigma_{m+m'};
&
\sigma'_m\circ(\phi_N)_{m'm}&=\phi_{m'm}\circ\sigma'_{m+m'}.
\end{align*}

\end{enumerate}
\end{proposition}
\begin{proof}
Let us recall the explicit definition of $\sigma_m$ given in \cite[Theorem 5.4.8]{EGHMW}. Let $
\wh p_{\R}:\cE^\bullet_U\left(\Im\frac{df}{f},m\right) \\ \rightarrow p_*\cE^\bullet_{U_N}\left(\Im\frac{df_N}{f_N},m\right)$ be the morphism of complexes of sheaves given by
\[
\begin{array}{lccc}
\wh p_{\R}:&\cE^j_U\otimes_k R_m&\longrightarrow& p_*\cE^j_{U_N}\otimes_k R(N)_m\\
&  ( \omega\otimes s^k) &\mapsto& p^* \omega \otimes \frac{(s_N)^k}{N^k}.
\end{array}
\]
in each degree $j$. Recall that the domain and target of this map are soft resolutions of $\ov{\cL}\otimes_R R_m$ and $p_*( \ovLn )\otimes_{R(N)}R(N)_m$ respectively through the maps defined like $\nu$  in Section~\ref{s:thickldr} (which we denote $\nu$ and $\nu_N$ respectively to distinguish them). $\sigma_m$ is the map induced by $\wh p_{\R}$ in this way, which is shown in the proof of \cite[Theorem 5.4.8]{EGHMW} to come from a map of mixed Hodge complexes (and therefore induce a MHS morphism in cohomology) and to be {a} right inverse of $\sigma'_m$. Property \eqref{part:sigmaMHC} appears in the proof of \cite[Theorem 5.4.8]{EGHMW}, and property \eqref{part:phipsi} can be easily checked through direct computation.

Property \eqref{part:Rlinear} is new, because $\cM_t$ has been defined in Lemma~\ref{lem:t-liftstoEquiv} of this paper. Since $\theta_{\ov{\cL_N}}\circ p_*\cM_t\circ \theta_{\ov{\cL_N}}^{-1}$ is multiplication by $t$ on $\ov{\cL}$, then property \eqref{part:Rlinear} for $\sigma'_m$ follows immediately from the definition of $\sigma'_m$. For $\sigma_m$, we have that for all $c\in\ov{\cL}\otimes_R R_m$,
\begin{align*}
\sigma_m(t\cdot c)&=\nu_N^{-1}\circ\wh p_{\R}\circ\nu(t\cdot c)=\nu_N^{-1}\circ\wh p_{\R}(\exp(2\pi s)\cdot\nu(c))= &\text{(By equation (\ref{eqn:nulinear}))}\\
&=\nu_N^{-1}\left(\exp\left(2\pi \frac{s_N}{N}\right)\cdot \wh p_{\R}(\nu(c))\right)= &\text{(By definition of }\wh p_{\R}\text{)}\\
&=\nu_N^{-1}\circ(p_*\widetilde \cM_t)\circ\wh p_{\R}\circ\nu(c)=
&\text{{(By definition of }}\widetilde \cM_t\text{)}
\\
&=(p_*\cM_t)\circ\nu_N^{-1}\circ\wh p_{\R}\circ\nu(c)=(p_*\cM_t)\circ\sigma_m(c).&\text{(By Lemma~\ref{lem:t-liftstodR})}
\end{align*}
\end{proof}

\begin{corollary}\label{cor:noCoverNeeded}
Let $\wt{(\Tors_R H^*(U;\ov{\cL}))_1}$ denote the MHS induced by the kernel of the map induced by $\psi_{mm}$ in cohomology explained in Remark~\ref{remk:noN} for $m\gg 0$, and let $\widehat{(\Tors_R H^*(U;\ov{\cL}))_1}$ denote the MHS induced as a subspace of $\Tors_R H^*(U;\ov{\cL})$. Then, both MHS coincide.
\end{corollary}
\begin{proof}
Let $N$ be as in Theorem~\ref{thm:summary}, part (\ref{part:N}), and let $m\gg 0$. By property \eqref{part:phipsi} in Proposition~\ref{prop:recallingIndepN}, $\sigma_m$ induces a morphism
$$
\begin{tikzcd}
\wt{(\Tors_R H^*(U;\ov{\cL}))_1}\arrow[r,hook,"\sigma", "\text{MHS}"']&\Tors_R H^*(U;\ov{\cL}).
\end{tikzcd}
$$
By property \eqref{part:Rlinear} in Proposition~\ref{prop:recallingIndepN}, $\sigma$ factors as follows:
$$
\begin{tikzcd}
\wt{(\Tors_R H^*(U;\ov{\cL}))_1}\arrow[r,"\sigma","\cong"']&\widehat{(\Tors_R H^*(U;\ov{\cL}))_1}\arrow[r,hook,"\text{MHS}"]&\Tors_R H^*(U;\ov{\cL}).
\end{tikzcd}
$$
Hence,
\begin{equation}\label{eqn:MHS}
\begin{tikzcd}
\wt{(\Tors_R H^*(U;\ov{\cL}))_1}\arrow[r,"\sigma","\cong"']&\widehat{(\Tors_R H^*(U;\ov{\cL}))_1}
\end{tikzcd}
\end{equation}
is a MHS isomorphism. We want to show that the identity $\wt{(\Tors_R H^*(U;\ov{\cL}))_1}\to\widehat{(\Tors_R H^*(U;\ov{\cL}))_1}$ is a MHS isomorphism, so it suffices to show that $\sigma$ is multiplication by a non-zero real constant.

By property \eqref{part:phipsi} in Proposition~\ref{prop:recallingIndepN}, $\sigma'_m$ induces a morphism $\sigma'$ such that the following composition is the identity.
$$
\begin{tikzcd}
\wt{(\Tors_R H^*(U;\ov{\cL}))_1}\arrow[r,hook,"\sigma", "\text{MHS}"']&\Tors_R H^*(U;\ov{\cL})\arrow[r,two heads,"\sigma'"]&{(\Tors_R H^*(U;\ov{\cL}))_1}.
\end{tikzcd}
$$

We claim that $\sigma': \Tors_R H^*(U;\ov{\cL})\rightarrow (\Tors_R H^*(U;\ov{\cL}))_1$ is, up to multiplication by $N$, the projection onto the generalized eigenspace. This will finish the proof, since it implies that {the map in equation} \eqref{eqn:MHS} is multiplication by a constant.

Let us verify this. Consider the maps induced by $\sigma'_m$ in cohomology, i.e. the result of taking cohomology in equation~\eqref{eqn:sigma'}. With the identifications that we have made to endow the torsion part of Alexander modules with a MHS (see Lemma~\ref{233} and Section~\ref{s:MHS}), $N\cdot\sigma'_m$ is the only arrow $\star$ that makes the following diagram commutative.
\[
\begin{tikzcd}[column sep = 4em]
\Tors_{R(N)} H^i(U_N; \ovLn )
\arrow[d,"\theta_N","\sim"']
\arrow[rr,"\eqref{ker}"',"\subseteq"]
&&
H^i(U_N; \ovLn \otimes_{R(N)} R(N)_m)\arrow[d,"\theta_N","\sim"']
\\
\Tors_R H^i(U;\ov\cL)
\arrow[r,"\subseteq"]
\arrow[d,"\star"]
&
H^i(U;\ov\cL)\otimes_{R(N)} R(N)_m
\arrow[r,"\mu_1"]
\arrow[d,"\mu_3"]
&
H^i(U;\ov\cL\otimes_{R(N)} R(N)_m)
\arrow[d,"a"]
\\
(\Tors_R H^i(U;\ov\cL))_1
\arrow[r,"\subseteq"]
\arrow[rr,bend right = 10,"\text{Remark \ref{remk:noN}}"]
&
H^i(U;\ov\cL)\otimes_{R} R_m
\arrow[r,"\mu_2"]
&
H^i(U;\ov\cL\otimes_{R} R_m)
\end{tikzcd}
\]
Here, the arrows labeled $\mu_1$ and $\mu_2$ are the result of factoring $H^i(U;\ov\cL)\to  H^i(U;\ov\cL\otimes_{R(N)} R(N)_m)
$ and $H^i(U;\ov\cL)\to  H^i(U;\ov\cL\otimes_{R} R_m)$ through the respective quotients. $\mu_3$ is induced from the inclusion $R(N)\subset R$ by tensoring with the identity. This ensures that $\mu_2\circ \mu_3 = a\circ \mu_1$. The arrows labeled ``$\subseteq$'' are the obvious arrows. Note that the composition of the bottommost row is the map from Remark~\ref{remk:noN}, so indeed this part of the diagram also commutes.

We just need to verify that the projection onto the generalized eigenspace of eigenvalue $1$ also makes that diagram commutative. Note that for any $R$-module $M$, the map $M\otimes_{R(N)} R(N)_m\to M\otimes_R R_m$ is the same as the quotient map $\frac{M}{(t^N-1)^mM}\to \frac{M}{(t-1)M}$, so indeed the bottom left hand square commutes if $\star$ is the projection. This concludes the proof.

\end{proof}

\section{Computations for formal varieties}\label{s:formal}

$(\Tors_R H^j(U;\ov\cL))_1$ is a direct summand of $\Tors_R H^j(U;\ov\cL)$ as a mixed Hodge structure. The goal of this section is to compute the MHS on $(\Tors_R H^j(U;\ov\cL))_1$ in some special cases.

By Theorem~\ref{thm:summary}, part~\eqref{part:cover}, the covering map induces a MHS morphism
$$
H_j(\pi):\Tors_R H_j(U^f;\Q)\rightarrow H_j(U;\Q).
$$
Let us take a look at the Milnor long exact sequence (coming from the corresponding short exact sequence at the chain complex level):
$$
\ldots\to H_j(U^f;\Q)\xrightarrow{\cdot(t-1)}H_j(U^f;\Q)\xrightarrow{H_j(\pi)}H_j(U;\Q)\to H_{j-1}(U^f;\Q)\to\ldots
$$
The exactness gives us the following.
\begin{corollary}\label{cor:2}
$H_j(\pi)$ is injective when restricted to $(\Tors_R H_j(U^f;\Q))_1$ if and only if the $R$-module $(\Tors_R H_j(U^f;\Q))_1$ is semisimple.
\end{corollary}

\begin{corollary}\label{cor:3}
Suppose that $f:U\rightarrow\C^*$ induces an epimorphism on fundamental groups, which implies that $U^f$ is connected. If $(\Tors_R H_j(U^f;\Q))_1$ is semisimple for all $j\leq r$, then the Milnor long exact sequence induces the following short exact sequences
$$
0\to(\Tors_R H_j(U^f;\Q))_1\oplus \frac{\operatorname{Free} H_j(U^f;\Q)}{(t-1)\operatorname{Free} H_j(U^f;\Q)}\xrightarrow{H_j(\pi)}H_j(U;\Q)\to (\Tors_R H_{j-1}(U^f;\Q))_1\to 0
$$
for all $j\leq r$. In particular, since $U^f$ is connected, $H_0(U^f;\Q)\cong \Q[t^{\pm 1}]/(t-1)$, and by induction we obtain
$$
\dim_{\Q}(\Tors_R H_j(U^f;\Q))_1=\sum_{l=0}^j (-1)^{l+j}\left(\dim_{\Q}H_l(U;\Q) - \rank_{\Q[t^{\pm 1}]}H_l(U^f;\Q)\right).
$$
\end{corollary}

In this section, we are interested in computing the MHS on the direct summand corresponding to the eigenvalue $1$ when the variety $U$ is formal (or more generally, $r$-formal, for $r\geq 1$). We start by recalling the definition in the case of real coefficients. The definition with $\Q$ coefficients used by Dupont in \cite{dupont} implies this one.
\begin{dfn}
A real manifold $U$ is \textit{formal} if the global sections of its real de Rham complex can be joined through a zig-zag of quasi-isomorphisms of cdga's to its cohomology algebra $H^*(U;\R)$ (as a complex with zero differential).
\end{dfn}

This definition can be generalized as follows.
\begin{dfn}
Let $r$ be an integer or $\infty$. An $r$-quasi-isomorphism is a morphism between chain complexes inducing maps in cohomology which are isomorphisms in degree $j$ for all $j\leq r$ and injective for $j=r+1$.

A real manifold is $r$\textit{-formal} if the global sections of its real de Rham complex can be joined through a zig-zag of $r$-quasi-isomorphisms of cdga's to a complex with zero differential. An $\infty$-formal manifold is a formal manifold.
\end{dfn}

We recall the following ``purity implies formality'' result of Dupont.

\begin{thm}[\cite{dupont}, Theorem 2.6]\label{thm:dupont}
Let $U$ be a smooth complex algebraic variety, $r$ a non-negative integer or $\infty$, and assume that one of the following conditions is satisfied:
\begin{enumerate}
\item for every integer $j\leq r+1$, $H^j(U;\Q)$ is pure of weight $j$;
\item\label{part:dupont2}for every integer $j\leq r$, $H^j(U;\Q)$ is pure of weight $2j$
\end{enumerate}
Then $U$ is $r$-formal.
\end{thm}

\begin{remk}\label{remk:11}
Note that $H^1(\C^*;\Q)$ is pure of weight 2. Therefore, if $f\colon U\to \C^*$ induces an epimorphism on fundamental groups, $f^*H^1(\C^*;\Q)$ is a sub-MHS of $H^1(U;\Q)$ of weight 2. For this reason, in this paper we will only be interested in case \eqref{part:dupont2} of Theorem~\ref{thm:dupont}.
\end{remk}

\begin{remk}[\cite{dupont}, Remark 2.8]\label{remk:dupont}
In the context of Theorem~\ref{thm:dupont}, case \eqref{part:dupont2}, the MHS on $H^j(U;\Q)$ can only be pure of type $(j,j)$ for all $j\leq r$, and therefore its dual, $H_j(U;\Q)$, is pure of type $(-j,-j)$.
\end{remk}

\begin{proposition}\label{prop:rqiso}
Suppose that $U$ is $r$-formal, with $r\geq 1$. Let $K^\bullet$ be a cdga with zero differential such that $K$ can be joined with $\Gamma(U,\cE_U^\bullet)$ through a zig-zag of $r$-quasi-isomorphisms. Let $\eta\in K^1$ corresponding to $\Im\frac{df}{f}$ through the isomorphism in $H^1$ that these $r$-quasi-isomorphisms induce. Then, $K^\bullet(\eta,m)$ can be joined with $\Gamma\left(U,\cE_U^\bullet\left(\Im\frac{df}{f},m\right)\right)$ through a zig-zag of $r$-quasi-isomorphisms for all $m\geq 1$, all of which are morphisms of complexes of $R$-modules.
\end{proposition}
\begin{proof}
This is a consequence of the theory of thickenings of cdga's developed in \cite[Section 3]{EGHMW}.
More specifically, it follows from combining the following two results:
\begin{enumerate}
\item If $G^\bullet$ is a cdga, and $\eta_1,\eta_2$ are cohomologous elements in $G^1\cap\ker d$, then $G(\eta_1,m)$ is isomorphic to $G(\eta_2,m)$ for all $m\geq 1$ as complexes of $R$-modules: \cite[Lemma 3.1.3]{EGHMW}. The isomorphism is given by $\exp(a\otimes s)\wedge -\ $, where $a\in G^0$ is such that $da=\eta_1-\eta_2$.
\item If $F:A_1^\bullet\rightarrow A_2^\bullet$ is an $r$-quasi-isomorphism and a morphism of cdga's, and $\eta\in A_1^1\cap\ker d$, then the map $F\otimes\id:A_1(\eta,m)\rightarrow A_2(F(\eta),m)$ is an $r$-quasi-isomorphism for all $m\geq 1$: \cite[Lemma 3.1.4]{EGHMW}. This lemma is stated for the case $r=\infty$ ($F$ is a quasi-isomorphism), but the same proof holds when $r$ is an integer greater or equal than $1$.
\end{enumerate}
\end{proof}

\begin{remk}\label{remk:compute}
Note that, since $\cE_U^\bullet\left(\Im\frac{df}{f},m\right)$ is a soft resolution of $\ov{\cL}_m$, then
$$
\Tors_R H^{j}(U;\ov\cL_m)\cong H^{j}\Gamma\left(U,\cE_U^\bullet\left(\Im\frac{df}{f},m\right)\right)$$ for all $m\geq 1$ and, by Remark~\ref{remk:noN},
$$\left(\Tors_R H^{j}(U;\ov\cL)\right)_1\cong\ker\left(H^{j}\Gamma\left(U,\cE_U^\bullet\left(\Im\frac{df}{f},m\right)\right)\xrightarrow{\cdot s^m}H^{j}\Gamma\left(U,\cE_U^\bullet\left(\Im\frac{df}{f}2,m\right)\right)\right)$$ for $m\gg 0$.

In particular, Proposition~\ref{prop:rqiso} tells us that, if $U$ is $r$-formal, then
$$
 H^{j}(U;\ov\cL_m)\cong H^{j}\left(U,K^\bullet\left(\eta,m\right)\right)$$ for all $m\geq 1$ and for all $j\leq r$, and
$$
\left(\Tors_R H^{j}(U;\ov\cL)\right)_1\cong\ker\left(H^{j}\left(U,K^\bullet\left(\eta,m\right)\right)\xrightarrow{\cdot s^m}H^{j}\left(U,K^\bullet\left(\eta,2m\right)\right)\right)$$ for $m\gg 0$, for all $j\leq r$.
\end{remk}

We now recall the following result of Budur, Liu and Wang from \cite{budurliuwang}, stated with the notation used in this paper. Originally in \cite{budurliuwang}, this was done in the case when $r=\infty$, but it remains true for integer values of $r$ by  Remark~\ref{remk:compute}.

\begin{proposition}[\cite{budurliuwang}, Proposition 1.7]\label{prop:BLW}
Let $U$ be a smooth connected complex algebraic variety which is $r$-formal for some integer $r\geq 1$ or $r=\infty$, and let $f:U\rightarrow \C^*$ be an algebraic map inducing an epimorphism on fundamental groups. Let $\cL$ be the rank $1$ local system of $R$-modules induced by $f$. Then, $(\Tors_R H^j(U;\ov\cL))_1$ is a semisimple $R$-module for all $j\leq r$.
\end{proposition}

We can now prove Corollary~\ref{cor:formal1} from Section~\ref{s:intro}. We restate it here for clarity, this time using cohomological notation.

\begin{corollary}\label{cor:formal}
Let $U$ be a smooth connected complex algebraic variety, let $f:U\rightarrow \C^*$ be an algebraic map inducing an epimorphism on fundamental groups, and let $r\geq 1$. Let $\cL$ be the rank $1$ local system of $R$-modules induced by $f$. Assume that 
 for every integer $k\leq r$, $H^k(U;\Q)$ is pure of weight $2k$. Then, for all $j\leq r-1$,
\begin{itemize}
\item the MHS summand $(\Tors_R H^{j+1}(U;\ov\cL))_1$ of $\Tors_R H^{j+1}(U;\ov\cL)$ is pure of type $(j,j)$.
\item $(\Tors_R H^{j+1}(U;\ov\cL))_1$ has dimension $$\sum_{l=0}^j (-1)^{l+j}\left(\dim_{\Q}H_l(U;\Q) - \rank_{\Q[t^{\pm 1}]}H_l(U^f;\Q)\right).$$
\end{itemize}
\end{corollary}

\begin{proof}
Under the hypotheses of this Corollary, Theorem~\ref{thm:dupont} and Proposition~\ref{prop:BLW} imply that $(\Tors_R H^{j+1}(U;\ov\cL))_1$ is semisimple for all $j\leq r-1$. By Corollary~\ref{cor:dual}, $(\Tors_R H_j(U^f;\Q))_1$ is also semisimple, and by Corollary~\ref{cor:2}, it is a sub-MHS of $H_j(U;\Q)$, which is pure of type $(-j,-j)$ by Remark~\ref{remk:dupont}. Hence, $(\Tors_R H_j(U^f;\Q))_1$ must also be pure of type $(-j,-j)$ and thus, by Corollary~\ref{cor:dual}, its dual MHS $(\Tors_R H^{j+1}(U;\ov\cL))_1$ is pure of type $(j,j)$. Moreover, the dimension of $(\Tors_R H_j(U^f;\Q))_1$ (which is the same as the dimension of $(\Tors_R H^{j+1}(U;\ov\cL))_1$) was computed in Corollary~\ref{cor:3} under the semisimplicity assumption. 
\end{proof}

The quantity $\rank_{\Q[t^{\pm 1}]}H_l(U^f;\Q)$ that appears in Corollary~\ref{cor:formal} above is defined using the free part of $H_l(U^f;\Q)$, but this paper is about the torsion part. For the reader who is more comfortable with rank $1$ $\C$-local systems and cohomology jump loci, we give an interpretation of $\rank_{\Q[t^{\pm 1}]}H_l(U^f;\Q)$ in terms of those in equations (\ref{eqn:local}) and (\ref{eqn:jump}) of the following remark. 
\begin{remk}[The quantity $\rank_{\Q[t^{\pm 1}]}H_l(U^f;\Q)$]\label{remk:rank}
Let $\lambda\in\C^*$, and let $L_{\lambda}$ be the corresponding rank $1$ $\C$-local system on $\C^*$. Let $\left(\Tors_R H_l(U^f;\C)\right)_{\lambda}$ be the kernel of multiplication by $t-\lambda$ on $\Tors_R H_l(U^f;\C)$, which by Theorem~\ref{thm:summary}, part (\ref{part:first}) is zero if $\lambda^N\neq 1$. Note also that  $\Tors_R H_l(U^f;\C)/(t-\lambda)\Tors_R H_l(U^f;\C)$ is zero if $\lambda^N\neq 1$. The generalized Milnor long exact sequence (\cite[Equation (4.3)]{Lib97})
$$
\ldots\to H_l(U^f;\C)\xrightarrow{\cdot(t-\lambda)}H_l(U^f;\C)\to H_l(U;f^{-1}L_{\lambda})\to H_{l-1}(U^f;\C)\to\ldots
$$
gives rise to the short exact sequence
$$
0\to \frac{\Tors_R H_l(U^f;\C)}{(t-\lambda)}\oplus\frac{\operatorname{Free} H_l(U^f;\C)}{(t-\lambda)}\to H_l(U;f^{-1}L_{\lambda})\to\left(\Tors_R H_{l-1}(U^f;\C)\right)_{\lambda}\to 0
$$

which tells us that
$$
\dim_{\C}H_l(U;f^{-1}L_{\lambda})\geq\rank_{\Q[t^{\pm 1}]}H_l(U^f;\Q)=\dim_{\C}\frac{\operatorname{Free} H_l(U^f;\C)}{(t-\lambda)\operatorname{Free} H_l(U^f;\C)},
$$
and it also tells us that equality holds for all but finitely many $\lambda$ (if $\lambda^N\neq 1$). In particular,
\begin{equation}\label{eqn:local}
\rank_{\Q[t^{\pm 1}]}H_l(U^f;\Q)=\min\{\dim_{\C}H_l(U;f^{-1}L_{\lambda})\mid \lambda\in\C^*\}.
\end{equation}
Similarly, let $M_B(X)$ be the set of rank $1$ $\C$-local systems on a topological space $X$, and let $\mathcal{V}^l_i(X)\coloneqq\{L\in M_B(X)\mid \dim_{\C}H^l(X;L)\geq i\}$ be the cohomology jump loci of $X$. Then,
\begin{equation}\label{eqn:jump}
\rank_{\Q[t^{\pm 1}]}H_l(U^f;\Q)=\max\{i\mid f^{-1}M_B(\C^*)\subseteq\mathcal{V}^l_i(U)\}.
\end{equation}
\end{remk}

Let us consider two illustrative examples in which the purity hypothesis of Corollary~\ref{cor:formal} holds, namely the ones mentioned in Remark~\ref{remk:pure}.

\subsection{Toric arrangements}\label{remk:tors}
Suppose that $U$ is an $n$-dimensional closed smooth connected subvariety of  $(\C^*)^r$, so in particular, $U$ affine and thus has the homotopy type of an $n$-dimensional CW complex. Furthermore, assume that $f:U\to\C^*$ induces a generic epimorphism on fundamental groups. Then $H_j(U^f;\Q)$ is a torsion $R$-module for all $j\leq n-1$, free for $j=n$ and $0$ otherwise (\cite[Corollary 1.4]{LMWtopology}). In particular, if $U$ is a toric arrangement complement in $(\C^*)^n$ and $f$ induces a generic epimorphism on fundamental groups, then, for all $j\leq n-1$,
$$
\dim_{\Q}(\Tors_R H_j(U^f;\Q))_1=\dim_{\Q}(\Tors_R H^{j+1}(U;\ov\cL))_1=\sum_{l=0}^j (-1)^{l+j}\dim_{\Q}H_l(U;\Q).
$$

\subsection{Hyperplane arrangements}\label{remk:hyperplanes}
Let $\cA$ be an affine hyperplane arrangement consisting of $d$ distinct hyperplanes in $\C^n$, let $U$ be its complement, and let $r$ be its rank, that is, the maximal codimension of a non-empty intersection of some subfamily of hyperplanes in the arrangement. Suppose moreover that $f=f_1^{\varepsilon_1}\cdot\ldots\cdot f_d^{\varepsilon_l}$, where the $f_i$'s are irreducible polynomials defining the $d$ distinct hyperplanes in the arrangement, and $\varepsilon_i$ are integers. If $\gcd\{\varepsilon_1,\ldots,\varepsilon_l\}=1$, and $\varepsilon_i\geq 1$ for all $i$, then (\cite[Theorem 4]{thesiseva},\cite[Proposition 6.1]{kohnopajitnov})$$H_j(U^f;\Q)\text{ is }\left\{\begin{array}{lr}\text{a torsion }R\text{-module}&\text{ for all }0\leq j\leq r-1\\
\text{a free }R\text{-module}&\text{for }j=r\\
0 & \text{otherwise,}\end{array}\right.
$$
and, as explained in \cite[Section 10.1]{EGHMW} (cf. \cite{thesiseva})
$$
\Tors_R H^{j+1}(U;\ov\cL)\cong\left\{\begin{array}{lr}H^j(U^f;\Q)& \text{ \ \ \ \ if }0\leq j\leq r-1\\0& \text{otherwise.}\end{array}\right.
$$

For simplicity in the notation, we will talk about the MHS on $H^j(U^f;\Q)$ from now on, and assume that $j$ is in the appropriate range. Corollary~\ref{cor:formal} provides the following extension of \cite[Theorem 10.1.5 (1)]{EGHMW}, which proved the case when $j=1$ and $\varepsilon_i=1$ for all $i$.
\begin{corollary}
Consider the notation and assumptions from the beginning of Section~\ref{remk:hyperplanes}. Then, for all $j\leq r-1$, $H^j(U^f;\Q)_1$ is the unique (up to isomorphism) pure Hodge structure of type $(j,j)$ of dimension $\sum_{l=0}^j (-1)^{l+j}\dim_{\Q}H_l(U;\Q)$. In particular, the MHS on $H^j(U^f;\Q)_1$ only depends on combinatorial data of the arrangements (the Betti numbers of $U$).
\end{corollary}

Let $H^j(U^f;\Q)_{\neq 1}$ be the direct sum of all the  summands other than $H^{j}(U^f;\Q)_1$ in the direct sum decomposition of Remark~\ref{remk:decompcohom}. The rest of this section is devoted to discussing what we know about the MHS on $H^j(U^f;\Q)_{\neq 1}$ when $U$ is an affine hyperplane arrangement complement, under the new hypothesis that $f$ is a reduced defining polynomial of the arrangement.

Recall that if $\cA$ is a central arrangement, then $U^f$ is homotopy equivalent to the global Milnor fiber $F$ of $f$. In that case, \cite[Theorem 7.7]{dimca2017hyperplanes} tells us that $\Gr^W
_{2j} H^j(F, \Q)_{\neq 1} = 0$, that is, the only possible non-zero graded pieces for the weight filtration are $j,j+1,\ldots,2j-1$. In \cite[Theorem 10.1.5]{EGHMW}, this result was generalized to non-central arrangement complements $U$, that is, to $H^j(U^f, \Q)_{\neq 1}$, but only for $j=1$. In that proof, the semisimplicity of $H^1(U^f, \Q)$ was heavily used (recall item (\ref{part:first}) after Corollary~\ref{cor:eigenspaces}). However, we do not know whether  $H^j(U^f, \Q)$ is semisimple for $j\geq 1$, this was stated as an open problem in \cite[Question 2]{EGHMW}. In fact, a generalization of \cite[Theorem 10.1.5]{EGHMW} to the case $j=2$ can only be achieved if $H^2(U^f, \Q)$ is semisimple, as exemplified by the following result.

\begin{corollary}
If $\Gr^W_4 H^2(U^f;\Q)_{\neq 1}\cong 0$, then $H^2(U^f;\Q)$ is a semisimple $R$-module.
\end{corollary}
\begin{proof}
By \cite[Theorem 7.4.1]{EGHMW}, $\Gr^W_i H^2(U^f;\Q)=0$ if $i\notin[2,4]$, so for $\Gr^W_i H^2(U^f;\Q)_{\neq 1}$, the same is true. Let $N$ and $m$ be such that $(t^N-1)^m$ annihilates $\Tors_R H_*(U^f;\Q)$ for all $*$. By part~\eqref{part:log} of Theorem~\ref{thm:summary}, multiplication by $\log(t^N)$ decreases the weight by $2$. Hence, if $\Gr^W_4 H^2(U^f;\Q)_{\neq 1}\cong 0$, then multiplication by $\log(t^N)$ is identically $0$ on $H^2(U^f;\Q)_{\neq 1}$.

Since $\log(t^N)$ (seen as a power series on $(t^N-1)$) and $(t^N-1)$ differ by a unit in $R(N)_m$ for all $m\geq 1$, we have that multiplication by $\log(t^N)$ is identically $0$ on $H^2(U^f;\Q)$ if and only if $H^2(U^f;\Q)$ is a semisimple $R$-module.  The result follows from the observation that multiplication by $\log(t^N)$ is identically $0$ on $H^2(U^f;\Q)_{1}$, since it is semisimple by Proposition~\ref{prop:BLW}. 
\end{proof}

\bibliographystyle{plain}

\bibliography{ambib}

\end{document}